\documentclass{amsart}
\usepackage{amsfonts}

\newtheorem{theorem}{Theorem}
\theoremstyle{plain}

\newtheorem{corollary}{Corollary}

\newtheorem{lemma}{Lemma}

\newtheorem{remark}{Remark}

\numberwithin{equation}{section}
\input{tcilatex}

\begin{document}
\title{Inequalities via $n$-times differentiable convex functions}
\author{Merve Avci Ardic}
\address{Adiyaman University, Faculty of Science and Arts, Department of
Mathematics, 02040, Adiyaman}
\email{mavci@adiyaman.edu.tr}

\begin{abstract}
In this paper, we establish some integral ineuqalities for $n-$ times
differentiable convex functions.
\end{abstract}

\maketitle

\section{introduction}

A function $\ f:I\rightarrow 
\mathbb{R}
$ is said to be convex function on $I$ if the inequality%
\begin{equation*}
f(\alpha x+(1-\alpha )y)\leq \alpha f(x)+(1-\alpha )f(y)
\end{equation*}%
holds for all $x,y\in I$ and $\alpha \in \lbrack 0,1].$

In the literature, Hermite-Hadamard inequality is known one of the most
significant inequality for convex functions. This inequality is defined
below:

Let $f:I\subset 
\mathbb{R}
\rightarrow 
\mathbb{R}
$ be a convex function on an interval $I$ and $a,b\in I$ with $a<b.$ Then 
\begin{equation*}
f\left( \frac{a+b}{2}\right) \leq \frac{1}{b-a}\int_{a}^{b}f(x)dx\leq \frac{%
f(a)+f(b)}{2}.
\end{equation*}%
For some inequalities, generalizations and applications including convex
functions and Hermite-Hadamard inequality it is possible to refer to \cite{1}%
-\cite{6} and \cite{10}, \cite{11}, \cite{14} and \cite{15}.

In \cite{15}, Dragomir and Agarwal proved following inequalities which
contain the right hand side of Hermite-Hadamard inequality.

\begin{theorem}
\label{teo 1.1} Let $f:I^{\circ }\subset 
\mathbb{R}
\rightarrow 
\mathbb{R}
$ be a differentiable mapping on $I^{\circ },$ $a,b\in I^{\circ }$ with $a<b.
$ If $\left\vert f^{\prime }\right\vert $ is convex on $[a,b],$ then the
following inequality holds:%
\begin{equation*}
\left\vert \frac{f(a)+f(b)}{2}-\frac{1}{b-a}\int_{a}^{b}f(x)dx\right\vert
\leq \frac{(b-a)(\left\vert f^{\prime }(a)\right\vert +\left\vert f^{\prime
}(b)\right\vert )}{8}.
\end{equation*}
\end{theorem}

\begin{theorem}
\label{teo 1.2} Let $f:I^{\circ }\subset 
\mathbb{R}
\rightarrow 
\mathbb{R}
$ be a differentiable mapping on $I^{\circ },$ $a,b\in I^{\circ }$ with $a<b,
$ and let $p>1.$ If the new mapping $\left\vert f^{\prime }\right\vert
^{p/(p-1)}$ is convex on $[a,b],$ then the following inequality holds:%
\begin{eqnarray*}
&&\left\vert \frac{f(a)+f(b)}{2}-\frac{1}{b-a}\int_{a}^{b}f(x)dx\right\vert 
\\
&\leq &\frac{b-a}{2(p+1)^{1/p}}\left[ \frac{\left\vert f^{\prime
}(a)\right\vert ^{p/(p-1)}+\left\vert f^{\prime }(b)\right\vert ^{p/(p-1)}}{2%
}\right] ^{(p-1)/p}.
\end{eqnarray*}
\end{theorem}

It is possible to find a number of manuscripts about $n-$ time
differentiability in the literature. For example, in \cite{6}, \cite{10} and 
\cite{11} the authors generalized Hermite-Hadamrd inequality for functions
whose derivatives of $n-$th order are $\left( \alpha ,m\right) -$convex, $s-$%
convex in the second sense and $m-$convex respectively. Pachpatte obtained
generalizations of Ostrowski and Trapezoid inequalities for $n-$times
differentiable functions in \cite{7}. Kechriniotis and Theodorou proved some
integral inequalities via $n-$times differentiable functions and gave some
applications for probability density function in \cite{9}. In \cite{12},
Cerone et.al proved some general inequalities of Ostrowski type and gave
applications for numerical integration and power series. In \cite{13},
Barnett and Dragomir proved Ostrowski type, perturbed type and trapezoid
type inequalities for $n-$times differentiable functions. Similarly, in \cite%
{8}, Sofo obtained general integral inequalities for $n-$times
differentiable functions.

The main aim of this paper is to establish some integral ineuqalities for
functions whose derivatives of $n-$th order are convex. 

To prove our main result we need the following lemma from \cite{13}.

\begin{lemma}
\label{lem 1.1} Let $f:[a,b]\rightarrow 
\mathbb{R}
$ be a mapping such that the derivative $f^{(n-1)}$ $(n\geq 1)$ is
absolutely continuous on $[a,b]$. Then for any $x\in \lbrack a,b]$ one has
the equality:%
\begin{eqnarray*}
\int_{a}^{b}f(t)dt &=&\sum_{k=0}^{n-1}\frac{1}{\left( k+1\right) !}\left[
\left( x-a\right) ^{k+1}f^{(k)}(a)+\left( -1\right) ^{k}\left( b-x\right)
^{k+1}f^{(k)}(b)\right]  \\
&&+\frac{1}{n!}\int_{a}^{b}\left( x-t\right) ^{n}f^{(n)}(t)dt.
\end{eqnarray*}
\end{lemma}

\section{inequalities for $n-$times differentiable convex functions}

Shall we start the following result.

\begin{theorem}
\label{teo 2.1} Let $f:[a,b]\rightarrow 
\mathbb{R}
$ be a mapping such that the derivative $f^{(n-1)}$ $(n\geq 1)$ is
absolutely continuous on $[a,b]$. If $\left\vert f^{(n)}\right\vert $ is
convex on $[a,b],$ following inequality holds%
\begin{eqnarray*}
&&\left\vert \int_{a}^{b}f(t)dt-\sum_{k=0}^{n-1}\frac{1}{\left( k+1\right) !}%
\left[ \left( x-a\right) ^{k+1}f^{(k)}(a)+\left( -1\right) ^{k}\left(
b-x\right) ^{k+1}f^{(k)}(b)\right] \right\vert  \\
&\leq &\frac{1}{n!(b-a)}\left\{ \left\vert f^{(n)}(a)\right\vert \left[ 
\frac{(b-x)(x-a)^{n+1}}{n+1}+\frac{(b-x)^{n+2}}{(n+1)(n+2)}+\frac{(x-a)^{n+2}%
}{n+2}\right] \right.  \\
&&+\left. \left\vert f^{(n)}(b)\right\vert \left[ \frac{(b-x)^{n+1}(x-a)}{n+1%
}+\frac{(x-a)^{n+2}}{(n+1)(n+2)}+\frac{(b-x)^{n+2}}{n+2}\right] \right\} 
\end{eqnarray*}%
for all $x\in \lbrack a,b]$.

\begin{proof}
From Lemma \ref{lem 1.1}, property of the modulus and convexity of $%
\left\vert f^{(n)}\right\vert ,$ it is possible to write%
\begin{eqnarray*}
&&\left\vert \int_{a}^{b}f(t)dt-\sum_{k=0}^{n-1}\frac{1}{\left( k+1\right) !}%
\left[ \left( x-a\right) ^{k+1}f^{(k)}(a)+\left( -1\right) ^{k}\left(
b-x\right) ^{k+1}f^{(k)}(b)\right] \right\vert  \\
&\leq &\frac{1}{n!}\int_{a}^{b}\left\vert x-t\right\vert ^{n}\left\vert
f^{(n)}(t)\right\vert dt \\
&=&\frac{1}{n!}\left\{ \int_{a}^{x}\left( x-t\right) ^{n}\left\vert
f^{(n)}(t)\right\vert dt+\int_{x}^{b}\left( t-x\right) ^{n}\left\vert
f^{(n)}(t)\right\vert dt\right\}  \\
&=&\frac{1}{n!}\left\{ \int_{a}^{x}\left( x-t\right) ^{n}\left\vert
f^{(n)}\left( \frac{b-t}{b-a}a+\frac{t-a}{b-a}b\right) \right\vert dt\right. 
\\
&&\left. +\int_{x}^{b}\left( t-x\right) ^{n}\left\vert f^{(n)}\left( \frac{%
b-t}{b-a}a+\frac{t-a}{b-a}b\right) \right\vert dt\right\}  \\
&\leq &\frac{1}{n!}\left\{ \int_{a}^{x}\left( x-t\right) ^{n}\left[ \frac{b-t%
}{b-a}\left\vert f^{(n)}(a)\right\vert +\frac{t-a}{b-a}\left\vert
f^{(n)}(b)\right\vert \right] \right.  \\
&&\left. +\int_{x}^{b}\left( t-x\right) ^{n}\left[ \frac{b-t}{b-a}\left\vert
f^{(n)}(a)\right\vert +\frac{t-a}{b-a}\left\vert f^{(n)}(b)\right\vert %
\right] \right\} .
\end{eqnarray*}%
If we use the equalities above in below, we get the desired result:%
\begin{equation*}
\int_{a}^{x}\left( x-t\right) ^{n}(b-t)dt=\frac{(b-x)(x-a)^{n+1}}{n+1}+\frac{%
(x-a)^{n+2}}{n+2}
\end{equation*}%
\begin{equation*}
\int_{a}^{x}\left( x-t\right) ^{n}\left( t-a\right) dt=\frac{(x-a)^{n+2}}{%
(n+1)(n+2)},
\end{equation*}%
\begin{equation*}
\int_{x}^{b}\left( t-x\right) ^{n}(t-a)dt=\frac{(b-x)^{n+1}(x-a)}{n+1}+\frac{%
(b-x)^{n+2}}{n+2}
\end{equation*}%
and%
\begin{equation*}
\int_{x}^{b}\left( t-x\right) ^{n}\left( b-t\right) dt=\frac{(b-x)^{n+2}}{%
(n+1)(n+2)}.
\end{equation*}
\end{proof}
\end{theorem}

\begin{corollary}
\label{co 1.1} In Theorem \ref{teo 1.1}, if $x=\frac{a+b}{2},$ following
inequality holds:%
\begin{eqnarray*}
&&\left\vert \int_{a}^{b}f(t)dt-\sum_{k=0}^{n-1}\frac{1}{\left( k+1\right) !}%
\left( \frac{b-a}{2}\right) ^{k+1}\left[ f^{(k)}(a)+\left( -1\right)
^{k}f^{(k)}(b)\right] \right\vert  \\
&\leq &\frac{(b-a)^{n+1}}{2^{n+1}(n+1)!}\left[ \left\vert
f^{(n)}(a)\right\vert +\left\vert f^{(n)}(b)\right\vert \right] .
\end{eqnarray*}
\end{corollary}

\begin{corollary}
\label{co 1.2} In Theorem \ref{teo 1.1}, if $n=1,$ following inequality
holds:%
\begin{eqnarray*}
&&\left\vert \int_{a}^{b}f(t)dt-\left\{ (x-a)f(a)+(b-x)f(b)\right\}
\right\vert  \\
&\leq &\left\vert f^{\prime }(a)\right\vert \left[ \frac{\left( b-x\right)
^{3}+\left( x-a\right) ^{2}\left( 3b-2a-x\right) }{6(b-a)}\right]  \\
&&+\left\vert f^{\prime }(b)\right\vert \left[ \frac{\left( x-a\right)
^{3}+\left( b-x\right) ^{2}\left( 2b+x-3a\right) }{6(b-a)}\right] .
\end{eqnarray*}
\end{corollary}

\begin{remark}
\label{rem 1.1} In Corollary \ref{co 1.2}, if we choose $x=\frac{a+b}{2},$
the inequality reduces to inequality in Theorem \ref{teo 1.1}.
\end{remark}

\begin{theorem}
\label{teo 2.2}  Let $f:[a,b]\rightarrow 
\mathbb{R}
$ be a mapping such that the derivative $f^{(n-1)}$ $(n\geq 1)$ is
absolutely continuous on $[a,b]$. If $\left\vert f^{(n)}\right\vert ^{q}$ is
convex on $[a,b],$ following inequality holds for all $x\in \lbrack a,b]$%
\begin{eqnarray*}
&&\left\vert \int_{a}^{b}f(t)dt-\sum_{k=0}^{n-1}\frac{1}{\left( k+1\right) !}%
\left[ \left( x-a\right) ^{k+1}f^{(k)}(a)+\left( -1\right) ^{k}\left(
b-x\right) ^{k+1}f^{(k)}(b)\right] \right\vert  \\
&\leq &\frac{\left( b-a\right) ^{1/q}}{n!}\left( \frac{\left( x-a\right)
^{np+1}+\left( b-x\right) ^{np+1}}{np+1}\right) ^{1/p}\left( \frac{%
\left\vert f^{(n)}(a)\right\vert ^{q}+\left\vert f^{(n)}(b)\right\vert ^{q}}{%
2}\right) ^{1/q}
\end{eqnarray*}%
where $p>1$ and $\frac{1}{p}+\frac{1}{q}=1.$
\end{theorem}

\begin{proof}
From Lemma \ref{lem 1.1}, property of the modulus, well-known H\"{o}lder
integral inequality and convexity of $\left\vert f^{(n)}\right\vert ^{q},$
we can write%
\begin{eqnarray*}
&&\left\vert \int_{a}^{b}f(t)dt-\sum_{k=0}^{n-1}\frac{1}{\left( k+1\right) !}%
\left[ \left( x-a\right) ^{k+1}f^{(k)}(a)+\left( -1\right) ^{k}\left(
b-x\right) ^{k+1}f^{(k)}(b)\right] \right\vert  \\
&\leq &\frac{1}{n!}\left( \int_{a}^{b}\left\vert x-t\right\vert
^{np}dt\right) ^{1/p}\left( \int_{a}^{b}\left\vert f^{(n)}(t)\right\vert
^{q}dt\right) ^{1/q} \\
&=&\frac{1}{n!}\left( \int_{a}^{x}\left( x-t\right)
^{np}dt+\int_{x}^{b}\left( t-x\right) ^{np}dt\right) ^{1/p}\left(
\int_{a}^{b}\left\vert f^{(n)}\left( \frac{b-t}{b-a}a+\frac{t-a}{b-a}%
b\right) \right\vert ^{q}dt\right) ^{1/q} \\
&\leq &\frac{1}{n!}\left( \frac{\left( x-a\right) ^{np+1}+\left( b-x\right)
^{np+1}}{np+1}\right) ^{1/p}\left( \int_{a}^{b}\left[ \frac{b-t}{b-a}%
\left\vert f^{(n)}(a)\right\vert ^{q}+\frac{t-a}{b-a}\left\vert
f^{(n)}(b)\right\vert ^{q}\right] dt\right) ^{1/q} \\
&=&\frac{\left( b-a\right) ^{1/q}}{n!}\left( \frac{\left( x-a\right)
^{np+1}+\left( b-x\right) ^{np+1}}{np+1}\right) ^{1/p}\left( \frac{%
\left\vert f^{(n)}(a)\right\vert ^{q}+\left\vert f^{(n)}(b)\right\vert ^{q}}{%
2}\right) ^{1/q}.
\end{eqnarray*}%
The proof is completed.
\end{proof}

\begin{remark}
\label{rem 1.2} In Theorem \ref{teo 2.2}, if we choose $n=1$ and $x=\frac{a+b%
}{2}$ we obtain 
\begin{eqnarray*}
&&\left\vert \frac{f(a)+f(b)}{2}-\frac{1}{b-a}\int_{a}^{b}f(t)dt\right\vert 
\\
&\leq &\left( \frac{2}{p+1}\right) ^{1/p}\frac{b-a}{4}\left( \left\vert
f^{\prime }(a)\right\vert ^{q}+\left\vert f^{\prime }(b)\right\vert
^{q}\right) ^{1/q}
\end{eqnarray*}%
which is the inequality in Theorem \ref{teo 1.2}.
\end{remark}

The following result holds for concave functions.

\begin{theorem}
\label{teo 2.3} Let $f:[a,b]\rightarrow 
\mathbb{R}
$ be a mapping such that the derivative $f^{(n-1)}$ $(n\geq 1)$ is
absolutely continuous on $[a,b]$. If $\left\vert f^{(n)}\right\vert ^{q}$ is
concave on $[a,b],$ following inequality holds for all $x\in \lbrack a,b]$%
\begin{eqnarray*}
&&\left\vert \int_{a}^{b}f(t)dt-\sum_{k=0}^{n-1}\frac{1}{\left( k+1\right) !}%
\left[ \left( x-a\right) ^{k+1}f^{(k)}(a)+\left( -1\right) ^{k}\left(
b-x\right) ^{k+1}f^{(k)}(b)\right] \right\vert  \\
&\leq &\frac{\left( b-a\right) ^{1/q}}{n!}\left( \frac{\left( x-a\right)
^{np+1}+\left( b-x\right) ^{np+1}}{np+1}\right) ^{1/p}\left\vert
f^{(n)}\left( \frac{a+b}{2}\right) \right\vert 
\end{eqnarray*}%
where $p>1$ and $\frac{1}{p}+\frac{1}{q}=1.$
\end{theorem}

\begin{proof}
From Lemma \ref{lem 1.1}, property of the modulus and well-known H\"{o}lder
integral inequality, we can write%
\begin{eqnarray*}
&&\left\vert \int_{a}^{b}f(t)dt-\sum_{k=0}^{n-1}\frac{1}{\left( k+1\right) !}%
\left[ \left( x-a\right) ^{k+1}f^{(k)}(a)+\left( -1\right) ^{k}\left(
b-x\right) ^{k+1}f^{(k)}(b)\right] \right\vert  \\
&\leq &\frac{1}{n!}\left( \int_{a}^{b}\left\vert x-t\right\vert
^{np}dt\right) ^{1/p}\left( \int_{a}^{b}\left\vert f^{(n)}(t)\right\vert
^{q}dt\right) ^{1/q}.
\end{eqnarray*}%
Let us write%
\begin{equation*}
\int_{a}^{b}\left\vert f^{(n)}(t)\right\vert ^{q}dt=\left( b-a\right)
\int_{0}^{1}\left\vert f^{(n)}(\lambda b+(1-\lambda )a\right\vert
^{q}d\lambda .
\end{equation*}%
Since $\left\vert f^{(n)}\right\vert ^{q}$ is concave on $[a,b]$, we obtain
the following inequality via Jensen inequality:%
\begin{eqnarray*}
&&\left( b-a\right) \int_{0}^{1}\left\vert f^{(n)}(\lambda b+(1-\lambda
)a\right\vert ^{q}d\lambda  \\
&=&\left( b-a\right) \int_{0}^{1}\lambda ^{0}\left\vert f^{(n)}(\lambda
b+(1-\lambda )a\right\vert ^{q}d\lambda  \\
&\leq &(b-a)\left( \int_{0}^{1}\lambda ^{0}d\lambda \right) \left\vert
f^{(n)}\left( \frac{1}{\int_{0}^{1}\lambda ^{0}d\lambda }\int_{0}^{1}\left(
\lambda b+(1-\lambda )a\right) d\lambda \right) \right\vert ^{q} \\
&=&(b-a)\left\vert f^{(n)}\left( \frac{a+b}{2}\right) \right\vert ^{q}.
\end{eqnarray*}%
If we combine all of results, we get the desired. The proof is completed.
\end{proof}


\begin{thebibliography}{99}
\bibitem{1} H. Kavurmaci, M. Avci and M. E. \"{O}zdemir, New inequalities of
Hermite-Hadamard type for convex functions with applications, Journal of
Inequalities and Applications, 2011, 2011:86.

\bibitem{2} M. \"{O}zdemir, H. Kavurmaci, A. Akdemir and M. Avci,
Inequalities for convex and $s-$convex functions on $\Delta =[a,b]\times
\lbrack c,d]$, Journal of Inequalities and Applications, 2012, 2012:20.

\bibitem{3} M. W. Alomari, M. E. \"{O}zdemir and H. Kavurmaci, On companion
of Ostrowski inequality for mappings whose first derivatives absolute value
are convex with applications, Miskolc Mathematical Notes, Vol. 13 (2012),
No. 2, pp. 233-248.

\bibitem{4} B.-Y. Xi and F. Qi, Some Hermite-Hadamard type inequalities for
differentiable convex functions and applications, Hacettepe Journal of
Mathematics and Statistics, Vol. 42 (3) (2013), 243-257.

\bibitem{5} B.-Y. Xi and F. Qi, Hermite-Hadamard type inequalities for
functions whose derivatives are of convexities, Nonlinear Functional
Analysis and Applications, Vol. 18, No. 2 (2013), pp. 163-176.

\bibitem{6} S.-P. Bai, S.-H. Wang and F. Qi, Some Hermite-Hadamard type
inequalities for $n-$time differentiable $\left( \alpha ,m\right) -$convex
functions, Journal of Inequalities and Applications 2012, 2012:267.

\bibitem{7} B. G. Pachpatte, New inequalities of Ostrowski and Trapezoid
type for $n-$time differentiable functions, Bull. Korean Math. Soc. 41
(2004), No. 4, pp. 633-639.

\bibitem{8} A. Sofo, Integral inequalities for $n-$ times differentiable
mappings, with multiple branches, on the $L_{p}$ norm, Soochow Journal of
Mathematics, Volume 28, No. 2, pp. 179-221, 2002.

\bibitem{9} A. I. Kechriniotis and Y. A. Theodorou, New integral
inequalities for $n-$time differentiable functions with applications for
pdfs, Applied Mathematical Sciences, Vol. 2, 2008, no. 8, 353 - 362.

\bibitem{10} W.-D. Jiang, D.-W. Niu, Y. Hua, and F. Qi, Generalizations of
Hermite-Hadamard inequality to $n$-time differentiable functions which are $%
s-$convex in the second sense, Analysis (Munich) 32 (2012), no. 3, 209--220.

\bibitem{11} S.-H. Wang, B.-Y. Xi and F. Qi, Some new inequalities of
Hermite-Hadamard type for $n-$time differentiable functions which are $m-$%
convex, Analysis 32, 247-262 (2012).

\bibitem{12} Cerone, P, Dragomir, SS, Roumelotis, J: Some Ostrowski type
inequalities for $n-$time differentiable mappings and applications.
Demonstr. Math. 32(4), 697-712 (1999).

\bibitem{13} N. S. Barnett and S. S. Dragomir, Applications of Ostrowski's
version of the Gr\"{u}ss inequality for trapezoid type rules, RGMIA Res.
Rep. Coll., 2002, Vol.5, Iss. 4.

\bibitem{14} C. E. M. Pearce and J. Pe\v{c}ari\'{c}, Inequalities for
differentiable mappings with application to special means and quadrature
formulae, Applied Mathematics Letters, 13 (2000), 51-55.

\bibitem{15} S. S. Dragomir and R. P. Agarwal, Two inequalities for
differentiable mappings and applications to special means of real numbers
and to trapezoidal formula, Applied Mathematics Letters, 11 (5) (1998),
91-95.
\end{thebibliography}
\end{document}